\documentclass[11pt]{amsart}
\usepackage{amsmath, amssymb}
\usepackage{amsfonts}
\usepackage{amsthm}
\usepackage{amsopn}
\usepackage{amscd}

\usepackage [all]{xy}

\input xy

\xyoption{arc}

\newcommand{\sO}{\mathcal{O}}
\newcommand{\sC}{\mathcal{C}}

\newcommand{\cinf}{\sC^\infty}

\newcommand{\sE}{\mathcal{E}}

\newcommand{\sL}{\mathcal{L}}
\newcommand{\sX}{\mathcal{X}}

\newcommand{\bR}{\mathbf{R}}

\newcommand{\bZ}{\mathbf{Z}}

\newcommand{\ruu}{\bR^{1|1}}
\newcommand{\rou}{\bR^{0|1}}

\newcommand{\tensor}{\otimes}
\newcommand{\comp}{\circ}

\theoremstyle{plain}
\newtheorem{thm}{Theorem}[section]
\newtheorem{prop}[thm]{Proposition}

\newtheorem{cor}[thm]{Corollary}

\newtheorem{lem}[thm]{Lemma}

\newtheorem{rem}[thm]{Remark}

\theoremstyle{definition}

\theoremstyle{remark}

\let\a\alpha
\let\b\beta
\let\g\gamma

\let\th\theta

\let\i\iota

\let\f\varphi
\let\o\omega
\let\G\Gamma
\let\D\Delta

\let\O\Omega

\let\na\nabla
\let\ra\rightarrow
\let\lra\longrightarrow
\let\ti\tilde

\newcommand{\vd}{\partial_\th+\th\partial_t}

\begin{document}
\title{$1|1$ Parallel Transport and Connections}
\author{Florin Dumitrescu}
\date{\today}
\maketitle

\begin{abstract} 

A vector bundle with connection over a supermanifold leads naturally to a notion of parallel transport along superpaths. In this note we show that {\it every} such parallel transport along superpaths comes form a vector bundle with connection, at least when the base supermanifold is a manifold.
\end{abstract}

\tableofcontents


\section{Introduction and Statement of Result}
Let $E$ be a $\bZ/2$-graded vector bundle over a compact manifold $M$. We consider a notion of parallel transport along superpaths in $M$, generalizing the notion of reparametrization-invariant parallel transport along paths in $M$, and show that it characterizes even (grading-preserving) connections over $M$.

Such a problem  is motivated by obtaining a characterization of supersymmetric one-dimensional topological field theories (abbreviated TFTs) over a manifold. This would extend the description of one-dimensional TFTs over a space $M$ as vector bundles with connection over $M$ in \cite{DST}. The equivalence between connections and usual parallel transport seems to be a classical fact, but only recently appears in print (see for example \cite{SW1} or \cite{D4}).

The basic concepts we work with in this paper involve the differential geometry of supermanifolds, and for an introduction to the theory of supermanifolds the reader is referred to Deligne and Morgan \cite{DM}. A quick survey on supermanifolds  can be found in \cite{HKST}. The notion of $1|1$ parallel transport that we use here appears in \cite{D1}.

Recall that a connection (a.k.a. covariant derivative) on a ($\bZ/2$-graded) vector bundle $E$ over $M$ is a first-order differential operator
$ \na: \G(M,E)\ra \G(M, T^*M\tensor E) $
satisfying Leibniz rule, i.e.
\[ \na(fs)= df\tensor s+f\na(s), \ \ f\in\cinf(M), \ \ s\in\G(M,E). \]
The connection is called {\it even} if it respects the $\bZ/2$-grading of the bundle $E$ (along vector fields on $M$).

Let $\partial_{t}$ denote the standard vector field on $\bR$, and  $D=\vd$ the standard odd vector field on $\ruu$.   $1|1$ \emph{parallel transport on $E$ over $M$} is defined by parallel transport along (families of) {\it paths} $\bR\times S\ra M$ (parametrized by supermanifolds $S$), as lifts of $\partial_t$, as well as parallel transport along (families of) {\it superpaths} $\ruu\times S\ra M$, as lifts of  $D$. There is a compatibility relation for parallel transport along paths and superpaths given by diagrams
\[ \xymatrix @C=4pc{ \ruu\times S \ar[rr]^{ \bar{c}} \ar[dr]_{q\times id} & & M, \\
&  \bR\times S \ar[ur]_{c} & } \]
so that a section $s$ along $c$ is $\partial_t$-parallel if and only if $s$ is $D$-parallel along $\bar{c}$ (the map  $q:\ruu\ra\bR$ stands for the obvious projection map). Note that a section $s$ of $\bar{c}^*E$ is of the form $s=s_1+\th s_2$, with $s_i\in \G(c^*E)$, and $s$ is $D$-parallel iff $s=s_1\in \G(c^*E)$, with $s_1$ a $\partial_t$-parallel section. 

The parallel sections along (super)paths should be chosen in such a manner that, when restricted to (super)intervals (see Subsection 3.4 of \cite{D1}), they give rise to linear isomorphisms between the fibers at the endpoints. The $1|1$-parallel transport is compatible under gluing of (super)paths, is the identity on constant (super)paths and is {\it invariant under reparametrization}. The last condition for transport along superpaths means that if $\f:\ruu\times S\ra \ruu\times S$ is a family of diffeomorphisms of $\ruu$ parametrized by $S$ that preserve the distribution determined by the standard vector field $D=\vd$ on $\ruu$, then a section $s\in \G(\ruu\times S; c^*E)$ is parallel along the superpath $c$ if and only if $s\f$ is parallel along the superpath $c\f$ (see Section 3 of \cite{D1} for more details). Similarly, reparametrization-invariance for paths means as usual that the parallel transport is invariant under precomposition by diffeomorphisms of $\bR$.

We should also require that the parallel transport is {\it natural} in the parame-trizing superspace $S$ for families of superpaths. This means that, given a superpath $c:\ruu\times S\ra M$ and a map $\f$ of supermanifolds $S'\ra S$, a section $s'$ is parallel along $c\comp(1\times\f)$ if and only if it is of the form $s'=s\comp(1\times\f)$, for $s$ a parallel section along the superpath $c$.

Note that the definition presented here does not differ from the one given in \cite{D1}, where we performed parallel transport {\it only} using superpaths. Indeed, the parallel transport along paths $c:\bR\times S\ra M$ can be obtained from parallel transport along superpaths $\bar{c}=(q\times 1)\comp c: \ruu\times S\ra M$ if we require that the parallel transport is {\it natural} with respect to the projection map $q:\ruu\ra \bR$; in other words, a section $s$ along $\bar{c}$ is parallel iff it is the pull-back of a section along the map $c$. The reason we prefer to make the parallel transport along paths explicit is that the flow of an even vector field on a supermanifold does not give rise to an interesting family of superpaths (an $\ruu$-action) unless the vector field is the square of an odd vector field on the supermanifold; henceforth we are obliged to consider its $\bR$-action (see Section 2.6 of \cite{D1} for the notion of flows of vector fields on supermanifolds).

One can observe that the invariance under reparametrization of the parallel transport along superpaths implies the invariance under reparametrization of paths. Indeed, any family of diffeomorphisms $\f$ of $\bR$ as below  lifts to a family of diffeomorphisms $\bar{\f}$ of $\ruu$ that preserve the distribution of $D$, and functoriality of parallel transport with respect to the map $q$ and invariance under $\bar{\f}$ imply the invariance of parallel transport under the map $\f$.
\[ \xymatrix@C=4pc{\ruu\times S \ar[rr]^{\bar{\f}} \ar[d]_q \ar@/^1pc/[ddr]^{\bar{c}\bar{\f}} & & \ruu\times S \ar[d]^q \ar@/_1pc/[ddl]_{\bar{c}}\\
\bR\times S \ar@{-->}[rr]_\f \ar[dr]_{c\f} & & \bR\times S. \ar[dl]^{c} \\
& M & } \]
Let us remark that not every family of diffeomorphisms of $\ruu$ descends to a family of diffeomorphisms of $\bR$. For this to happen, the even part of the family $\bar{\f}$ should be independent of the odd variable $\theta$ on $\rou$.

The data given by a $1|1$-parallel transport map can be encoded as a smooth representation of the category $1|1$-tbord($M$) of $1|1$-topological bordisms over the manifold $M$ (whose objects are points in $M$, and morphisms are superpaths in $M$), i.e. a $1|1$-transport map defines a $1|1$-TFT over $M$. Topological field theories were first introduced by Atiyah in \cite{A} as a junction point between topology and quantum field theory. A modern approach to field theories using the language of categories can be found for example in  \cite{ST2}, \cite{HKST} or \cite{HST}.

The above notion of parallel transport can be word-for-word extended to vector bundles over supermanifolds. In \cite{D1} we show that a connection on a vector bundle over a supermanifold gives rise to such a parallel transport. This paper is concerned with showing the equivalence of the two notions when the base space is a manifold. This is enough if we are only interested in describing $1|1$-TFTs over a \emph{manifold}. 
Our main result is the following

\begin{thm}  \label{0} There is a natural 1-1 correspondence

\[  \left\{
\begin{array}{l}
1|1 \text{  parallel transport}\\
\text{ \ \ \ on $E$ over $M$}
\end{array} \right\}
\longleftrightarrow \left\{
\begin{array}{l}
\text{ Even connections  }\\
\text{ \ \  on $E$ over $M$}
\end{array} 
\right\}. \]

\end{thm}

In other words, since  $1|1$-TFTs over $M$ are represented by  $1|1$-parallel transport maps over $M$, we can reformulate the theorem as
\[  
1|1\text{-TFT}(M)
\cong \left\{
\begin{array}{l}
\bZ/2\text{-bundles with even connections  }\\
\text{\ \ \ \ \ \ \ \ \ \ \ \ \ \ \ \  over $M$}
\end{array} 
\right\}, \]
where the left-hand side denotes the space of all $1|1$-TFTs over the space $M$ (to avoid set-theoretic issues, we require that the field theories over points are vector spaces in a \emph{fixed} infinite dimensional  
vector space). 

In \cite{D1} Section 3 we constructed a map
\[  \left\{
\begin{array}{l}
1|1 \text{  parallel transport}\\
\text{ \ \ \ \ on $E$ over $M$}
\end{array} \right\}
\longleftarrow \left\{
\begin{array}{l}
\text{ Even connections  }\\
\text{ \ \ \ on $E$ over $M$}
\end{array} 
\right\}, \]

\noindent when $M$ is a supermanifold. The parallel transport along superpaths was defined by considering sections that are constant along superpaths in the direction of the vector field $D$ with respect to the pull-back connection. Moreover, the parallel transport {\it recovers} the connection, making the above map injective. The novelty of the Theorem \ref{0} consists in producing a connection out of a $1|1$-parallel transport map and show that it provides an inverse to the natural map ``$\longleftarrow$". The proof will be the result of the equivalences expressed in the diagram below.
\[ \xymatrix @  C=2cm { \{1|1\text{ transport on $E$ } \} \ar@{<->}[d]_{\text{Prop } \ref{2}}  \ar@{<-->}[r]^-{\text{Th } \ref{0}} & \{\text{even connections on $E$  }  \} \\
\{  1|1 \text{ o.t. transport on $\pi^*E$ }  
\} \ar@{<->}[r]_{\text{Prop } \ref{3}} & \{ \text{o.t. connections on $\pi^*E$  }  \}. \ar@{<->}[u]_{\text{Prop } \ref{1}} }\]
The bundle $\pi^*E$ is the pull-back bundle of the bundle $E$ via the map $\pi:\Pi TM\ra M$ from the ``odd tangent bundle'' of $M$ to $M$, which on functions is the inclusion of functions on $M$, as 0-forms, into the space of differential forms on $M$. The abbreviation o.t. stands for ``odd-trivial" (see below). 

The reason we run our proof through the intermediate supermanifold $\Pi TM$ instead of directly working on the manifold $M$ is that on $\Pi TM$ we can find  interesting {\it families} of superpaths to capture the geometry of $1|1$-parallel transport, namely those parametrized by $\Pi TM$ itself, expressing the flows of {\it odd} vector fields on $\Pi TM$. Alas, on $M$ there are no odd vector fields.


\section{Odd-trivial connections}

\begin{prop} \label{1} Let $E$ be a $\bZ/2$-graded vector bundle over $M$. There is a 1-1 correspondence

\[  \left\{
\begin{array}{l}
\text{ Grading-preserving connections  }\\
\text{\ \ \ \ \ \ \ \ \ \ \ on $E$ over $M$}
\end{array} \right\}
\longleftrightarrow \left\{
\begin{array}{l}
\text{Odd-trivial connections}\\
\text{ \ \ on $\pi^*E$ over $\Pi TM$}
\end{array} 
\right\} \]

\end{prop}
To motivate the definition of odd-trivial connections let us begin by stating the following lemma whose proof is clear.
\begin{lem} \label{ot}
Let $\ti{\na}$ denote a connection on the pullback bundle $\pi^*E$ over $\Pi TM$ via the projection map $\pi: \Pi TM \ra M$. Then 
\[ \ti{\na}= \pi^*\na, \]
for some connection $\na$ on the bundle $E$ over $M$ if and only if 
\[ \langle\ti{\na}(\pi^*s), \iota_X\rangle=0\ \ \text{ and }\ \  \langle\ti{\na}(\pi^*s), \sL_X\rangle\in\pi^*\G(M;E), \]
for any  $s\in\G(M;E)$ a section of $E$ and any  $X$ a vector field on $M$.
\end{lem}
\noindent Here $\iota_X$ is the contraction by the vector field $X$ acting on $\O^*(M)=\cinf(\Pi TM)$, interpreted as an odd derivation (i.e. vector field) on $\Pi TM$. Similarly, $\sL_X$ acts as a derivation in the direction of $X$ on differential forms, and is interpreted as an even vector field on $\Pi TM$. The sharp bracket stands for the pairing between 1-forms and vector fields on $\Pi TM$. \\

\noindent {\bf Remark:} The zero-equality above is not true for all odd vector fields on $\Pi TM$, for example we have
\[ \langle\ti{\na}(\pi^*s), d\rangle= \na s, \]
where $d$ is the standard odd vector field on $\Pi TM$, inducing the exterior derivative $d$ on differential forms. Still it stays true for vector fields pointing in the ``odd'' directions. (Note that $d$ can be written locally as $d=\sum dx^{i}\frac{\partial}{\partial x^{i}}$ so it points ``even''.)

Let us call such connections on pullback bundles $\pi^*E\ra \Pi TM$ as in Lemma \ref{ot} {\it odd-trivial connections}.

The proof of Proposition \ref{1} is now clear since the lemma is a mere reformulation of the statement. From Lemma \ref{ot} immediately follows 

\begin{lem} If $\ti{\na}$ is an odd-trivial connection, then $\ti{\na}$ is flat in the odd directions, i.e.
\[ [\ti{\na}_X, \ti{\na}_Y]= \ti{\na}_{[X,Y]}, \]
for $X, Y$ odd derivations on $\Pi TM$. 
\end{lem} 
\begin{proof} It is enough to check the relation for odd derivations of the type $\iota_X$, where $X$ is a vector field on $M$, since arbitrary odd derivations on $\Pi TM$ can be written as $\Omega^*(M)^{ev}$-combination of these. 

\end{proof}


\section{Odd-trivial $1|1$-parallel transport}

We say that the $1|1$-parallel transport on a bundle $\pi^*E$ over $\Pi TM$ is {\it odd-trivial} if the parallel transport along maps
\[ \bar{\a}_X:\ruu\times \Pi TM \ra \Pi TM, \]
given by the flow of vector fields (see Section 2.6 of \cite{D1}) of the form $\iota_X$ on $\Pi TM$, where $X$ is a vector field on $M$, is the identity on sections with initial condition of the form $\pi^*s\in\G(\pi^*E)$, for $s\in\G(E)$. Recall (see \cite{DM}) that for such odd vector fields $\iota_X$ on $\Pi TM$ that square to zero, the flow is actually determined by an $\rou$-action on $\Pi TM$, $\a_{\iota_X}:\rou\times \Pi TM \ra \Pi TM$ and the map $\bar{\a}_X$ factors as below
\[ \xymatrix @ C=2cm{ \ruu\times \Pi TM \ar[rr]^{ \bar{\a}_X} \ar[dr]_{p\times id} & & \Pi TM, \\
&  \rou\times \Pi TM \ar[ur]_{\a_{\iota_X}} & } \]
where $p:\ruu\ra\rou$ is the obvious projection map. The identity requirement above makes sense since the pullback of the bundle $\pi^*E$ via the map $\bar{\a}_X$ is the bundle $\ruu\times \pi^*E$ over $\ruu\times \Pi TM$, as the bundle is the pullback bundle of the bundle $\a_{\iota_X}^*\pi^*E$ via the map $p\times 1$, and the bundle $\a_{\iota_X}^*\pi^*E$ is the pullback bundle of $E$ via the map $p_0\times \pi:\rou\times \Pi TM \ra M$, i.e. it is the bundle $\rou\times \pi^*E$.  

We should also require that for  parallel transport along paths given by the flows $\a_X:\bR\times \Pi TM\ra \Pi TM$ of even vector fields $\sL_X$ on $\Pi TM$ coming from vector fields $X$ on $M$, we have that 
\[ p^{\Pi TM}(\a_X;\pi^*s\in \G(\pi^*E))\in (1\times\pi)^*\G(\underline{\a}_X^*E), \]
where the map $\underline{\a}_X:\bR\times M\ra M$ is the flow of the vector field $X$ on $M$. (We use the notation $p^N(c;s_0)$ for parallel sections in the space $N$ along the (super)path $c$, determined by the initial condition $s_0$.) Note that there is a compatibility of the flows with the projection map $\pi$, as illustrated by the diagram
\[ \xymatrix @ C=3cm{ \bR\times \Pi TM \ar[d]_{1\times \pi} \ar[r]^{\a_X} &  \Pi TM \ar[d]^\pi \\
\bR \times M \ar[r]_{\underline{\a}_X} & M. } \]

\begin{prop}  \label{2} There is a 1-1 correspondence

\[  \left\{
\begin{array}{l}
1|1 \text{  parallel transport}\\
\text{ \ \ \ on $E$ over $M$}
\end{array} \right\}
\longleftrightarrow \left\{
\begin{array}{l}
1|1 \text{ odd-trivial parallel transport}\\
\text{ \ \ \ \ \ \ on $\pi^*E$ over $\Pi TM$}
\end{array} 
\right\} \]

\end{prop}

\begin{proof}
``$\longleftarrow$" Denote by $j:\bR^0\ra \ruu$ the standard inclusion of a point in $\ruu$, namely mapping to $(0,0)\in \ruu$. Consider an arbitrary superpath $c$ in $M$ as below
\[ \xymatrix @C=4pc { c^*E \ar[r] \ar[d] & E \ar[r] \ar[d] & \pi^*E \ar[d] \\
\ruu\times S \ar[r]_c & M \ar[r]_i & \Pi TM.} \]
To define a $1|1$ parallel transport in $M$, we need to specify for each such superpath $c$ in $M$ a parallel section $p^M(c; h\tensor s)$  along $c$, for each initial condition 
$$h\tensor s\in \G(S, c_0^*E)\cong\cinf(S)\tensor \G(M, E),$$
where $h\in\cinf(S)$ and $s\in\G(M, E)$, and $c_0=c\comp j$. Define
\[ p^M(c; h\tensor s):= p^{\Pi TM}(ic; h\tensor s), \]
where $i:M\ra \Pi TM$ denotes the standard inclusion. Note that 
\[ c_0^*E\cong c_0^*i^*\pi^*E,\]
since $\pi i=id$. Let now $\f:\ruu\times S\ra \ruu\times S$ denote a family of diffeomorphisms of $\ruu$ preserving the conformal structure (the distribution determined by the vector field $D=\vd$ defining the standard metric structure on $\ruu$) and the point $(0,0)$. Then
\begin{eqnarray*}
p^M(c\f; h\tensor s) &=&  p^{\Pi TM}(ic\f; h\tensor s) \\
&= & p^{\Pi TM}(ic; h\tensor s)\comp\f \\
&= & p^M(c; h\tensor s)\comp\f.
\end{eqnarray*}
The second equality holds since the $1|1$-parallel transport on $\Pi TM$ is invariant under reparametrization. This means that the $1|1$-parallel transport on $M$ we constructed is invariant under reparametrization. Compatibility under glueing of superpaths and the identity on constant superpaths are obvious properties of the constructed parallel transport. 

Similarly, for a (family of) path(s) $c$ in $M$ as below
\[ \xymatrix @C=4pc { c^*E \ar[r] \ar[d] & E \ar[r] \ar[d] & \pi^*E \ar[d] \\
\bR\times S \ar[r]_c & M \ar[r]_i & \Pi TM,} \]
we define  $p^M(c; h\tensor s):= p^{\Pi TM}(ic; h\tensor s)$, for $h\tensor s\in \G(S, c_0^*E)$ a section along $c_0:S\ra M$. It is clear that the parallel transport along paths is invariant under reparametrization and compatible under glueing of paths. \\

``$\lra$" Given a superpath $c$ in $\Pi TM$ as below
\[ \xymatrix @C=4pc { c^*\pi^*E \ar[r] \ar[d] & \pi^*E \ar[r] \ar[d] & E \ar[d] \\
\ruu\times S \ar[r]_c & \Pi TM \ar[r]_\pi & M} \]
we need to specify a parallel section $p^{\Pi TM}(c; h\tensor \o\tensor s)$  along $c$ with initial condition
\begin{eqnarray*}
h\tensor\o\tensor s & \in & \G(S, c_0^*\pi^*E)\\
& \cong& \cinf(S)\tensor_{\cinf(\Pi TM)} \G(\Pi TM, \pi^*E)\\
&\cong & \cinf(S)\tensor_{\cinf(\Pi TM)} \cinf(\Pi TM)\tensor_{\cinf(M)} \G(M, E),
\end{eqnarray*}
where $h\in\cinf(S)$, $\o\in\cinf(\Pi TM)\cong\O^*(M)$ and $s\in\G(M, E).$ The map $c_{0}:S\ra \Pi TM$ denotes the restriction of $c$ to $(0,0)\times S$. We define such a parallel section by
\[ p^{\Pi TM}(c; h\tensor \o\tensor s):= p^M(\pi c; c_0^*(\o)h\tensor s), \]
As before, we check that 
\begin{eqnarray*}
p^{\Pi TM}(c\f; h\tensor \omega \tensor s) &=&  p^{M}(\pi c\f; c_0^*(\o)h\tensor s) \\
&= & p^{M}(\pi c; c_0^*(\o) h\tensor s)\comp\f \\
&= & p^{\Pi TM}(c; h\tensor\o\tensor s)\comp\f,
\end{eqnarray*}
for $\f$ an arbitrary family of diffeomorphisms of $\ruu$ preserving the conformal structure and the point $(0,0)$. The second equality holds since the $1|1$-parallel transport on $M$ is invariant under reparametrization. This means that the $1|1$-parallel transport on $\Pi TM$ we constructed is invariant under reparametrization. Compatibility under glueing of superpaths and the identity on constant superpaths are as before obvious. Parallel transport along paths in $\Pi TM$ is dealt with in a similar manner.

We are left to check the odd-triviality of the $1|1$ parallel transport. Let $\bar{\a}_X:\ruu\times \Pi TM\ra \Pi TM$ the flow of the odd vector field $\i_X$ on $\Pi TM$, for $X$ a vector field on $M$. Then 
\begin{eqnarray*}
p^{\Pi TM}(\bar{\a}_X; \pi^*s\in\G(\pi^*E)) &=&  p^{M}(\pi \bar{\a}_X; \pi^*s\in\G(\pi^*E)) \\
&= & \bar{\a}_X^*\pi^*s,
\end{eqnarray*}
since the map $\bar{\a}_X$ factors through $\a_{\i_X}:\rou\times \Pi TM\ra \Pi TM$, and the composition $\pi \bar{\a}_X$ is the uninteresting projection map.\\ 

Now, it is not hard to see that if we apply the construction  ``$\lra$"and then the construction ``$\longleftarrow$", we obtain the identity. To see that the correspondence in the Proposition is one-to-one, we are left to check that the construction ``$\longleftarrow$" is injective. This is a consequence of the following diagram 
\[ \xymatrix{ \{1|1\text{-odd trivial transport in } \Pi TM \} \ar@{<->}[d]_{\text{Prop \ref{3} (to be proven)}} \ar[r] & \{1|1\text{-transport in M} \} \\
\{\text{ odd-trivial connections over } \Pi TM \} \ar@{<->}[r] & \{ \text{even connections over } M \} \ar[u]
}\]
as well as the diagram 
\[ \xymatrix  @C=.2pc { \{ \text{even connections over } M \} \ar[rr] \ar[dr] & & \{1|1\text{-transport in M} \} \ar[dl]\\
& \{1\text{-transport in M} \}  & } \]
being commutative. Now observe that the lower right arrow map in the last diagram is injective since a connection is recovered by its usual parallel transport. Therefore the right arrow map is injective  
(a direct argument of this fact can be also found in Subsection 3.3 of \cite{D1}). This further implies, by looking back at the first diagram, the required injectivity. We conclude that the two constructions are inverses of one another, and so obtain the Proposition.

\end{proof}


\section{An odd-trivial equivalence}

\begin{prop}  \label{3} There is a 1-1 correspondence

\[  \left\{
\begin{array}{l}
1|1 \text{ odd-trivial parallel transport}\\
\text{ \ \ \ \ \ \ on $\pi^*E$ over $\Pi TM$}
\end{array} \right\}
\longleftrightarrow \left\{
\begin{array}{l}
\text{Odd-trivial connections}\\
\text{\ \ \  on $\pi^*E$ over $\Pi TM$}

\end{array} 
\right\} \]

\end{prop}

``$\longleftarrow$" It is clear how a connection gives rise to $1|1$-parallel transport: given a superpath $c$ in $\Pi TM$,  pull-back the connection along $c$ and define a section to be parallel along $c$ if it is constant in the direction of the vector field $D$ on $\ruu$.  Moreover, the odd-triviality of the connection implies the odd-triviality of the resulting parallel transport. 

We spend the remaining of this Section going in the other direction ``$\longrightarrow$" and end up showing that the two arrows are inverse to each other. We start off by lifting the action of vector fields of the type $\sL_X$ and $\iota_X$ on $\Pi TM$, for $X$ vector fields on $M$,  to actions on the total space of the bundle $\pi^*E$, which by differentiation gives us a compatible (under summation and function multiplication of vector fields) family of derivations, i.e. a connection on $\pi^* E$. 
In order to lift such actions we make some preliminary remarks on flows of vector fields in Subsection \ref{flows}, which are of independent interest, and then combine the even-odd rules of Subsection \ref{eorules} to obtain the algebraic properties of a connection.


\subsection{Remarks on flows of vector fields} \label{flows}
In this subsection we find a Trotter type formula relating the flow of the sum of two vector fields $X$ and $Y$, in terms of the flows of $X$ and $Y$, as well as a relation between the flow of $X$ and the flow of $fX$, for $f$ a function on the manifold. There is a definite advantage to express geometrically these algebraic operations from a field theoretic perspective. 

\begin{prop} Let $X$ and $Y$ be vector fields on $M$, and let $\a, \b:\bR\times M\ra M$ denote the flows determined by $X$, respectively $Y$. Then the flow $\g$ of the vector field $X+Y$ is given by 
\[ \g_t(x)= \lim_{n\to\infty}  \underbrace{(\a_{\frac{t}{n}}\b_{\frac{t}{n}})\comp \ldots \comp (\a_{\frac{t}{n}}\b_{\frac{t}{n}})}_{n}(x). \]

\end{prop}

\begin{proof} Let us begin with a calculation:

\begin{eqnarray*}
\frac{d}{dt}\Big |_{t=0}(\a_t\comp\b_t)(x) &=& \frac{d}{dt}\Big |_{t=0} \a(t,\b(t,x))\\
&=& \frac{\partial}{\partial t}\Big |_{t=0} \a(t,x)+\sum_{i=1}^n\frac{\partial \a}{\partial x^i}(0,x)
 \frac{\partial}{\partial t}\Big |_{t=0} \b^i(t,x)\\
 &=& (X+Y)(x).
\end{eqnarray*}
By a similar calculation, we have
\[ \frac{d}{dt}\Big |_{t=0}(\a_{\frac{t}{2}}\b_{\frac{t}{2}})\comp  (\a_{\frac{t}{2}}\b_{\frac{t}{2}})(x)= (X+Y)(x), \]
and more generally
\[ \frac{d}{dt}\Big |_{t=0} \underbrace{(\a_{\frac{t}{n}}\b_{\frac{t}{n}})\comp \ldots \comp (\a_{\frac{t}{n}}\b_{\frac{t}{n}})}_{n}(x)= (X+Y)(x), \]
for any $n$. Next, we will show the group property for the family $\{\g_t\}$. To simplify notation, denote $\underbrace{f\comp \ldots \comp f}_n$ by $f^{(n)}$. We then have

\begin{eqnarray*}
\g_{2t} &=& \lim_{n\to\infty} (\a_{2t/2n}\b_{2t/2n})^{(2n)}\\
&=& \lim_{n\to\infty} (\a_{t/n}\b_{t/n})^{(n)} (\a_{t/n}\b_{t/n})^{(n)}\\
&=& \g_t\g_t.
\end{eqnarray*} 
By a similar calculation, we obtain $\g_{3t}=\g_t\g_t\g_t$, and more generaly
\[ \g_t= \g_{t/n}^{(n)}, \text{ for all } n\geq 1. \]
This implies that
\[ \g_t\g_s= \g_{t+s}, \]
for all $t,s$ rational numbers, and, by continuity, for all $t,s$ real numbers. \\
Note that the limit in the statement of the proposition exists, as one can check for example by a Taylor expansion in $t$, for a fixed $x\in M$, and verifying that the Taylor coefficients converge.

\end{proof}

\begin{rem} A word-for-word translation of the proof above shows that the same result holds for $X$ and $Y$ even vector fields on a compact supermanifold $M$.
\end{rem}

Consider now $X$  a vector field on a (compact) manifold $M$. This determines an odd vector field $\iota_X$ on $\Pi TM$ that squares to zero. Its flow is reduced to a map $\a:\rou\times\Pi TM\ra \Pi TM$ given by 
\[ \a^*:\O^*M\ra \O^*M[\th]:\ \  \o\mapsto \o+(\iota_X\o)\th. \]

\begin{lem}
Let $X$ and $Y$ be vector fields on $M$ and $\iota_X$, $\iota_Y$ the corresponding odd vector fields on $\Pi TM$ with flow maps $\rou\times\Pi TM\ra \Pi TM$ denoted by $\a$ and $\b$. Then the flow $\g$ of $\ \iota_X+\iota_Y$ is given by 
\[ \g:\rou\times\Pi TM\ra \Pi TM,\ \ \g=\b\comp(1\times\a)\comp(\D\times 1), \]
where $\D:\rou\ra\rou\times\rou$ is the diagonal map. On $S$-points, this means
\[ \g(\th,x)=\b(\th,\a(\th,x)). \]
\end{lem}

\begin{proof}
We have to check that the following diagram commutes
\[ \xymatrix@C=4pc{\rou\times \Pi TM \ar[r]^-\g \ar[d]_{\D\times 1} & \Pi TM \\
\rou\times\rou\times \Pi TM  \ar[r]_-{1\times\a}   &   \rou\times \Pi TM. \ar[u]_\b } \]
This, on functions, translates into commutativity of the diagram
\[ \xymatrix @C=5pc{  \o+\iota_X\o\th+\iota_Y\o\th & \o \ar@{|->}[l]_-{\g^*} \ar@{|->}[d]^{\b^*} \\
\o+\iota_X\o\th_2+(\iota_Y\o+\iota_X\iota_Y\o\th_2)\th_1\ar@{|->}[u]^{\th_1=\th_2} & \o+\iota_Y\o\th_1. \ar@{|->}[l]^-{\a^*} } \]
\end{proof}

\begin{rem} \label{ree}
The same proof shows that if $X$ and $Y$ are two odd vector fields on a supermanifold that square to zero and their Lie bracket $[X,Y]$ is also zero, then the sum $X+Y$ is an odd vector field that squares to zero and  its flow (an $\rou$-action) is the composition of the flows of $X$ and $Y$.
\end{rem}

\vspace{.1cm}

\begin{lem}  Let $\a:\bR\times M\ra M$ be the flow of a vector field $X$ on the compact manifold $M$. If $f$ is a positive function on $M$ then the flow of $fX$ is given by $$\b:\bR\times M \ra M: (t,x)\mapsto \a(s(t,x),x),$$ where $s:\bR\times M \ra \bR$ is the solution to

$$\left\{ \begin{array}{l} 
\frac{\partial s}{\partial t}(t,x)= f(\a(s(t,x),x))
\\\\
s(0,x)=0, \text{ for all } x.

\end{array} \right.$$

\end{lem}

\noindent The proof is a routine check.

\begin{cor}  Let $X$ and $Y$ be vector fields on $M$. Then $X$ and $Y$ have the same (directed) trajectories if and only if $Y=fX$, for some positive function $f$ on $M$. 

\end{cor}

\begin{cor} If $Y=fX$, for some positive function $f$ on $M$, and $c$ is an integral curve of $X$ then $c\comp \f$ is an integral curve of $Y$, for some (orientation-preserving) diffeomorphism $\f$ of $\bR$.
\end{cor}

When $M$ is a supermanifold, the situation is more involved. We still have as before

\begin{lem} \label{ee} Let $\a:\bR\times M\ra M$ be the flow of an \textnormal{even} vector field $X$ on the compact supermanifold $M$. If $f$ is a positive \textnormal{even} function on $M$ then the flow of $fX$ is given by $$\b:\bR\times M \ra M: (t,x)\mapsto \a(s(t,x),x),$$ where $s:\bR\times M \ra \bR$ is the solution to

$$\left\{ \begin{array}{l} 
\frac{\partial s}{\partial t}(t,x)= f(\a(s(t,x),x))
\\\\
s(0,x)=0, \text{ for all } x.

\end{array} \right.$$

\end{lem}

Let now $f$ be a positive {\it even} function and $X$ be an {\it odd} vector field with flow $\a:\ruu\times M\ra M$ on the supermanifold $M$. Let $\f:\ruu\times M\ra \ruu$ be a family of diffeomorphisms of $\ruu$ parametrized by $M$ that preserves the 1-dimensional distribution determined by the vector field $D$ on $\ruu$ so that
\[ (D\tensor 1)\comp \f^*=M_{f\a(\f\times1)(1\times\D)}\comp\f^*\comp D. \]
Here $f\a(\f\times1)(1\times\D):\ruu\times M\ra \ruu$ is an even function on $\ruu\times M$, and $M_g$ denotes multiplication by the function $g$. Then we have the following

\begin{lem} \label{eo} The flow of the odd vector field $fX$ is given by the map 
$$\b:\ruu\times M\ra M, \ \ \b=\a(\f\times1)(1\times \D), $$ 
or, on $S$-points,
\[ \b(t,\th,x)= \a(\f(t,\th,x),x). \]
\end{lem}

\begin{proof}
This is just a calculation. We have to check that 
\[(D\tensor 1)\comp\b^*= \b^*\comp fX. \]
Now
\begin{eqnarray*}
LHS &=& (D\tensor 1)\comp (1\tensor \D^*)\comp(\f^*\tensor 1)\comp \a^* \\
&=& (1\tensor \D^*)\comp(D\tensor 1\tensor 1)\comp(\f^*\tensor 1)\comp \a^*\\
&=& (1\tensor \D^*)\comp((D\tensor 1)\comp\f^*)\tensor 1\comp \a^*\\
&=& (1\tensor \D^*)\comp(M_{f\a(\f\times1)(1\times\D)}\comp\f^*\comp D)\tensor 1\comp \a^*\\
&=& M_{f\a(\f\times1)(1\times\D)}\comp(1\tensor \D^*)\comp((\f^*\comp D)\tensor 1)\comp \a^*.
\end{eqnarray*}
In the fourth equality we used the defining property of the family $\f$ of diffeomorphisms of $\ruu$. On the other hand,
\begin{eqnarray*}
RHS &=& (1\tensor \D^*)\comp(\f^*\tensor 1)\comp\a^*\comp fX \\
&=& M_{f\a(\f\times1)(1\times\D)}\comp ((1\tensor \D^*)\comp(\f^*\tensor 1)\comp\a^*\comp X )    \\ 
&=& M_{f\a(\f\times1)(1\times\D)}\comp ((1\tensor \D^*)\comp(\f^*\tensor 1)\comp(D\tensor 1)\comp\a^* ) \\
&=& M_{f\a(\f\times1)(1\times\D)}\comp((1\tensor \D^*)\comp((\f^*\comp D)\tensor 1)\comp \a^*),
\end{eqnarray*}
where in the third equality we used the fact that $\a$ is the flow of the vector field $X$. The two expressions coincide, and this verifies the lemma.

\end{proof}


\subsection{Even-odd rules} \label{eorules}
Consider the following families of even vector fields on $\Pi TM$
\[ \bar{e}= \{\ \sL_X\ | \ X \text{ vector field on } M\ \}, \]
respectively odd vector fields on $\Pi TM$
\[ \bar{o}= \{\ \iota_X\ | \ X \text{ vector field on } M\ \}. \]
The following lemma is easy to check.
\begin{lem}
\[ e\cdot \bar{o}\oplus o\cdot\bar{e} =\sX(\Pi TM)^\text{odd} \]
\[ e\cdot \bar{e}\oplus o\cdot\bar{o} =\sX(\Pi TM)^\text{ev}, \]
where $e$ and $o$ denote even, respectively odd functions on $\Pi TM$.
\end{lem}
To define a connection on $\pi^*E$ over $\Pi TM$ from an odd-trivial parallel transport, we first define
$\na_{\bar{e}}$ and $\na_{\bar{o}}$, using the flows of these vector fields and differentiating the parallel sections along these paths to obtain horizontal lifts, along which we differentiate arbitrary sections. 
It is not hard to check that 
\[ \na_{\bar{e}+\bar{e}}=\na_{\bar{e}}+\na_{\bar{e}}, \]
\[ \na_{\bar{o}+\bar{o}}=\na_{\bar{o}}+\na_{\bar{o}}. \]
This is true since in both cases we can express the flow of the sum of two vector fields in terms of the flows of each of the vector fields. Using the Lemmas \ref{ee} and \ref{eo}, we can check that 
\[ \na_{e\cdot\bar{e}}= e\cdot\na_{\bar{e}}, \]
\[ \na_{e\cdot\bar{o}}= e\cdot\na_{\bar{o}}. \]
We then define
\[ \na_{o\cdot\bar{o}}:= o\cdot\na_{\bar{o}}, \]
\[ \na_{o\cdot\bar{e}}:= o\cdot\na_{\bar{e}}. \]
If $\sE$ and $\sO$ denote the even, respectively odd vector fields on $\Pi TM$, we define
\[ \na_{\sE+\sE}:= \na_\sE+\na_\sE, \]
\[ \na_{\sO+\sO}:= \na_\sO+\na_\sO, \]
\[ \na_{\sE+\sO}:= \na_\sE+\na_\sO. \]
The first two relations require a consistency check. First, if $\sum \o_j\iota_{X_j}=0$, then 
\[ \na_{\sum \o_j\iota_{X_j}}=0, \]
since $\na_{\sum \o_j\iota_{X_j}}$ acts as the derivation $\sum \o_j\iota_{X_j}$ on $\G(\Pi TM; \pi^*E)=\O^*(M)\tensor \G(M; E)$. Second, if 
\[ \sum \o_j\sL_{X_j}=0, \]
then the  ${X_j}$'s are $\cinf(M)$-linearly dependent, and the two ways of defining for example
$\na_{\sL_{fX}}=\na_{f\sL_X}$, for $f\in\cinf(M)$, are consistent with each other.

We can summarize the above considerations in the following

\begin{lem} \label{eo}
Consider the map
\[ V\in\sX(\Pi TM) \longmapsto \na_V:\G(\Pi TM; \pi^*E)\ra \G(\Pi TM; \pi^*E), \]
so that $\na_{\bar{e}}$ and $\na_{\bar{o}}$ are $\bar{e}$- respectively $\bar{o}$-derivations. Moreover, we require
\[ \na_{e\cdot \bar{e}}= e\cdot\na_{\bar{e}}, \ \ \ \na_{e\cdot \bar{o}}= e\cdot\na_{\bar{o}}, \ \ \  \na_{\bar{e}+\bar{e}}=\na_{\bar{e}}+\na_{\bar{e}},\ \ \ \na_{\bar{o}+\bar{o}}=\na_{\bar{o}}+\na_{\bar{o}}. \]
Then $\na$ defines a connection on $\pi^*E$ over $\Pi TM$, extending by linearity the above relations.
\end{lem}

\subsection{Conclusion of the proof of Proposition \ref{3}}
Now we can finally describe the arrow ``$\longrightarrow$" of Proposition \ref{3} since an odd-trivial parallel transport defines a map $\na$ satisfying the conditions in the Lemma \ref{eo}, so $\na$ defines a connection on $\pi^*E$ over $\Pi TM$. This connection is clearly odd-trivial. Let us remark that the Lie bracket of odd vector fields lifts in a compatible way which is consistent with the fact that an odd-trivial connection is flat in the odd directions. 

The only thing left to check is that the two arrows are inverse of each other. A standard argument (see \cite{D1} Subsection 3.3) shows that the parallel transport of a connection recovers the connection; this means for us that the map ``$\longleftarrow$" of Proposition \ref{3} is injective. To finish the proof, it is enough to verify that
\[ \longleftarrow\comp\longrightarrow \ = id. \]

That is, start with a $1|1$ parallel transport on $E$ over $M$ and consider the associated connection $\na$ defined  by the even rules above. We have to verify that the $1|1$ parallel transport determined by the connection coincides with the $1|1$-transport we started off with. 


Recall that the connection $\na$ is defined by looking at parallel sections along families of (super)paths
\[ \ruu\times \Pi TM\ra \Pi TM, \ \ \ \  \bR\times \Pi TM\ra \Pi TM, \]
which come from flows of odd respectively even vector fields on $\Pi TM$.

First, using the definition of the connection and the properties of flows in Section 4.1, we infer that the two parallel transport functors coincide for flows of vector fields of the form
\[ \bar{e}\ \ \ \   \bar{o}\ \ \ \ \bar{e}+\bar{e}\ \ \ \  \bar{o}+\bar{o}\ \ \ \ e\cdot \bar{e}\ \ \ \ e\cdot \bar{o}, \]
as by $\bar{o}$ we denoted odd vector fields of the type $\iota_{X}$ on $\Pi TM$ (for $X$ a vector field on $M$) and these square to zero, and sums of these also square to zero. 

Next we verify the identity of the parallel transport functors on vector fields of the type
\[  o\cdot\bar{o}. \]
Let $f$ be an odd function and $X$ an odd vector field on $\Pi TM$ so that $X^{2}=0$ and $X(f)=0$ or $X(f)=1$ (This is no restriction, as the more general case $o\cdot\bar{o}$ is obtained by multiplication by an even function for which we apply Lemma  \ref{ee}). When $X(f)=0$, the flow of $fX$ is given by 
\[ \a:\bR\times\Pi TM \ra \Pi TM,\ \ \ \a^{*}\o=\o+t fX(\o). \]
Therefore $\a$ can be written as the composition
\[ \xymatrix@C=4pc{\bR\times \Pi TM \ar@{-->}[r]^{\a} \ar[d] & \Pi TM \\
\rou\times\bR\times \Pi TM  \ar[r]   &   \bR\times \Pi TM. \ar[u] } \]
where the first map is induced by the odd function $f$ on $\Pi TM$, the second map expresses the flow of the odd vector field $tX$ on $\bR\times M$ and the last map is the projection map. 

When $X(f)=1$, the flow of $fX$ is given by 
\[ \a:\bR\times\Pi TM \ra \Pi TM,\ \ \ \a^{*}\o=\o+(e^{t}-1) fX(\o). \]
Therefore $\a$ can be written as the composition
\[ \xymatrix@C=4pc{\bR\times \Pi TM \ar@{-->}[r]^{\a} \ar[d] & \Pi TM \\
\rou\times\bR\times \Pi TM  \ar[r]   &   \bR\times \Pi TM. \ar[u] } \]
where the first map is induced by the odd function $f$ on $\Pi TM$, the second map expresses the flow of the odd vector field $(e^{t}-1)X$ on $\bR\times M$ and the last map is the projection map.

In both situations we can express the parallel sections along $fX$ with respect to the parallel sections along $X$. This proves the compatibility of parallel transports along flows of the type $o\cdot\bar{o}$. The case $o\cdot\bar{o}+o\cdot\bar{o}$ as well as the cases $e\cdot\bar{e}+e\cdot\bar{e}$ and  $e\cdot\bar{e}+o\cdot\bar{o}$ are covered by Remark \ref{ree}. Therefore, the two parallel transport functors coincide alongside flows of arbitrary \emph{even} vector fields on $\Pi TM$.

Consider now  vector fields of the type  
\[o\cdot \bar{e}.\]
In this case we pass to the \emph{intermediate} space $\rou\times \Pi TM$ via the projection map 
\[ \rou\times \Pi TM \ra \Pi TM, \]
and use the functoriality of parallel transport under pullbacks. Let therefore $f$ be an odd function on $\Pi TM$ and $X$ an even vector field on $\Pi TM$ so that $X(f)=0$. This means in particular that $(fX)^{2}=0$. (This is no restriction on the type $o\cdot \bar{e}$, if we combine our choice with the fact shown below that the parallel transports coincide for flows of vector fields of the type $o\cdot \bar{e}+o\cdot \bar{e}$, so that the coefficient function is annihilated by the vector field).
The flow of $fX$ is then given by 
\[ \a:\rou\times\Pi TM \ra \Pi TM,\ \ \ \a^{*}\o=\o+\theta fX(\o). \]
This precise formula allows us to write the flow $\a$ as the composition

\[ \xymatrix@C=4pc{\rou\times \Pi TM \ar@{-->}[r]^{\a} \ar[d] & \Pi TM \\
\bR\times\rou\times \Pi TM  \ar[r]   &   \rou\times \Pi TM. \ar[u] } \]
where the first map is the inclusion $t=1$, the second map gives the flow of the {\it even} vector field $(\theta f)X$ on $\rou\times\Pi TM$ and the last map is the projection map. 
Because parallel transport is functorial and is the same along the $\bR$-action map in the diagram above, which expresses the flow of an \emph{even} function multiplying an even vector field, the two transport functors give rise to the same $\rou$-action, i.e. they coincide along the flow of $fX$.

A similar trick applies for vector fields of the type
\[  e\cdot\bar{o}+e\cdot\bar{o}\hspace{.5in} o\cdot\bar{e}+o\cdot\bar{e} \hspace{.5in}  e\cdot\bar{o}+o\cdot\bar{e} \]
by passing again to the intermediate space $\rou\times \Pi TM$ via the projection map 
\[ \rou\times \Pi TM \ra \Pi TM. \]
Indeed, let $X$ be a vector field of one of the three types above. The flow of $X$ is given by 
\[ \a:\ruu\times\Pi TM \ra \Pi TM,\ \ \ \ \ \ \ \a^{*}\o=e^{-tX^2+\th X}\o=e^{-tX^2}(1+\th X)\o. \]
The flow $\a$ can be written as the composition
\[ \xymatrix@C=4pc{\bR\times\rou\times \Pi TM \ar@{-->}[rr]^{\a} \ar[d] & & \Pi TM \\
\rou\times \Pi TM  \ar[r]   &   \bR\times\rou\times \Pi TM\ar[r] & \rou\times\Pi TM. \ar[u] } \]
The first map expresses the flow of the even vector field $-X^{2}$ on $\rou\times\Pi TM$, the second map is the inclusion $t=1$, the third map is the action map of the even vector field $\th X$ on $\rou\times\Pi TM$ and the last map is the projection map. Functoriality of parallel transport combined with the compatibility of the parallel transport functors for \emph{even} vector fields already proven, gives the required compatibility along the flow $\a$.

These cases cover all types of vector fields we have on $\Pi TM$. Therefore, the above composition is the identity for parallel transport along families determined by flows of \emph{arbitrary} vector fields on $\Pi TM$. 

Finally, we are left to consider an \emph{arbitrary family} of (super)paths and show that the given $1|1$ parallel transport coincides with the one emerging from the connection. This general situation reduces as follows to the case of families of super(paths) coming from flows of vector fields discussed so far.

By smoothness of parallel transport, the two transport functors coincide for {\it families of flows} of vector fields on $\Pi TM$, i.e. maps of the form 
\[ \ruu\times \Pi TM\times S \ra \Pi TM, \ \ \ \  \bR\times \Pi TM\times S\ra \Pi TM, \]
where $S$ is an arbitrary supermanifold. Consider now an arbitrary family of superpaths
\[ c:\ruu\times S \ra \Pi TM, \]
parametrized by a supermanifold $S$. Then $c$ factors as below
\[ \xymatrix@C=4pc{\ruu\times S \ar[rr]^c \ar[dr]_{1\times c_0\times 1} & & \Pi TM, \\
& \ruu\times \Pi TM\times S \ar[ur]^\a & } \]
where $c_0$ denotes the restriction of $c$ to ${(0,0)}\times S\hookrightarrow \ruu\times S$, and the upper right arrow $\a$ is a map of families of flows of vector fields on $\Pi TM$ parametrized by $S$ (such a map $\a$ exists since any superpath in a supermanifold is an integral curve of a vector field on the supermanifold, at least locally). Now, the parallel transport of the connection coincides with the original parallel transport along the superpath $\a$, and, by naturality of parallel transport, along $c$ as well. A similar argument applies for (families of) paths in $\Pi TM$. This then verifies that the above composition of arrows is the identity and concludes the proof of Proposition \ref{3}. 

This also finishes the proof of Theorem \ref{0}, by putting together Propositions \ref{1}, \ref{2} and \ref{3}. \\

One should add a word about the functoriality of the parallel transport via the map
\[ \rou\times \Pi TM \ra \Pi TM, \]
used in the proof above to treat the case of flows of arbitrary odd vector fields on the supermanifold $\Pi TM$.
As we noted, the two parallel transports coincide by definition for $\rou$-actions and specific $\bR$-actions, which allows us to show they coincide for \emph{all} $\bR$-actions, i.e. for the flows of all \emph{even} vector fields. This in particular implies that the two parallel transports coincide for arbitrary families of \emph{paths} ($\bR$-families), by the same argument that passes from flows to arbitrary families. If we take the parametrizing space for the family to be $\rou\times \Pi TM$, we obtain that the pullbacks of the transport functors under the map above coincide for $\bR$-actions on the auxiliary space $\rou\times \Pi TM$. Similar reasoning applies to functoriality of parallel transport  via the map 
\[ \bR\times \Pi TM\ra \Pi TM \]
relative to $\rou$-actions, used in identifying the parallel transport functors for flows of vector fields  of the type $o\cdot\bar{o}$. \\

\noindent\textbf{Concluding remarks.} 
\begin{enumerate}
\item A consequence of our proof above is that the flow of a sum of two odd vector fields on a supermanifold can be expressed as the composition of the flows of each of the vector fields, although we do not know a closed formula. From a field theoretic perspective this carries no weight, as the algebraic operation of summation is described via composition. The philosophical meaning is that whatever construction a field theory does over a space, it is carried on also to the algebraic operation of summation over the space. Similarly about function multiplication.\\

\item The main difficulty in proving this result was on circumventing the fact that we don't have a closed formula for the flow of sum of \emph{odd} vector fields (as we do for even vector fields), or for the flow of an \emph{odd} function multiplying a vector field. The formulas in Section 4.1 are sufficient though to allow for this ``even passage'' in the superworld. As a result we conclude that the information of a $1|1$ parallel transport over a (super)manifold $M$ is encoded in parallel transport along flows of even vector fields, and odd vector fields that square to zero. Thus we do not need to look at arbitrary families of (super)paths in $M$ to single out the transport functor.\\

\item As a final comment, let us remark that the map ``$\longleftarrow$" in Theorem \ref{0} coincides with our original construction of Section 3 in \cite{D1}. Indeed, if $c:\ruu\times S\ra M$ denotes a superpath in $M$, a parallel section along $c$ is given by a parallel section $s$ along $ic$, where $i:M\ra \Pi TM$ is the standard inclusion map, according to Proposition \ref{2}. By Proposition \ref{3}, $s$ must satisfy the differential equation
\[ (ic)^*(\pi^*\na)_D s=0. \]
As $\pi i=id$, we have that $i^*\pi^*\na=\na$, and therefore the above equation is equivalent to 
\[ (c^*\na)_D s=0, \]
the equation that defines the parallel transport along $c$ in \cite{D1}.
\end{enumerate}

\vspace{.4cm}

\noindent\emph{Acknowledgements.} The author would like to especially thank the referee for signaling to us a few important omissions in the first version of the paper and for suggestions that improved the presentation of the paper. This article was written while the author was visiting Max Planck Institute for Mathematics in Bonn and the Institute of Mathematics ``Simion Stoilow" in Bucharest, and completed while at University of Hamburg. We gratefully acknowledge these Institutions for their hospitality and support.

\bibliographystyle{plain}
\bibliography{bibliografie}

\bigskip
\raggedright Universit\"at Hamburg\\  Bundesstra\ss{}e 55 \\
D-20146 Hamburg\\ Email: {\tt florinndo@gmail.com}

\end{document}